\documentclass[a4paper,12pt]{amsart}
\usepackage[a4paper, margin=2.8cm]{geometry}

\usepackage{amssymb,amsmath,amsrefs}
\usepackage{enumerate,hyperref}
\tolerance
\hbadness
\vfuzz \hfuzz
\usepackage{comment}

\def\dep{\operatorname{depth}}

\def\supp{\operatorname{supp}}
\def\ass{\operatorname{Ass}}
\def\ann{\operatorname{Ann}}

\def\coht{\operatorname{coht}}
\def\hte{\operatorname{ht}}
\def\bight{\operatorname{bight}}
\def\reg{\operatorname{reg}}
\def\pdim{\operatorname{pdim}}
\def\deg{\operatorname{deg}}
\def\mdim{\operatorname{indim}}

\def\del{\operatorname{del}}

\def\link{\operatorname{link}}

\newcommand{\lr}[1]{\left\langle#1\right\rangle}

\newcommand{\ceil}[1]{\left\lceil #1 \right\rceil}

\newcommand{\f}{\mathcal{F}}
\newcommand{\n}{\mathcal{N}}

\newcommand{\K}{\mathbb{K}}

\newtheorem{theorem}{Theorem}[section]
\newtheorem{lemma}[theorem]{Lemma}
\newtheorem{corollary}[theorem]{Corollary}
\newtheorem{proposition}[theorem]{Proposition}

\theoremstyle{definition}
\newtheorem{definition}[theorem]{Definition}
\newtheorem{remark}[theorem]{Remark}
\newtheorem{example}[theorem]{Example}

\begin{document}
	
\title[Initially Cohen--Macaulay Modules]{Initially Cohen--Macaulay Modules}

\author{Mohammed Rafiq Namiq}
\address{Mohammed Rafiq Namiq, Department of Mathematics, College of Science, University of Sulaimani, Kurdistan Region, Iraq.}
\email{mohammed.namiq@univsul.edu.iq}

\keywords{Initially Cohen--Macaulay, Depth, Projective dimension, Simplicial complex, Edge ideal, Chordal graph, Maximum degree.}
\subjclass[2020]{Primary 13C70, 13C15, 13F55; Secondary 05E40, 	05C90, 05C07.}
%\date{\today} 

\maketitle 

\begin{abstract}
	In this paper, we introduce initially Cohen--Macaulay modules over a commutative Noetherian local ring $R$, a new class of $R$-modules that generalizes both Cohen--Macaulay and sequentially Cohen--Macaulay modules. A finitely generated $R$-module $N$ is initially Cohen--Macaulay if its depth is equal to its initial dimension, an invariant defined as the infimum of the coheights of the associated primes of $N$. We develop the theory of these modules, providing homological, combinatorial, and topological characterizations and confirming their compatibility with regular sequences, localization, and dimension filtrations. When this theory is applied to simplicial complexes, we establish analogues of Reisner’s criterion, the Eagon--Reiner theorem, and Duval’s characterization of sequentially Cohen--Macaulay complexes. Finally, we classify certain classes of initially Cohen--Macaulay graphs of interest and those whose projective dimension coincides with their maximum vertex degree.
\end{abstract}

\section{Introduction}
\quad For nearly a century, Cohen--Macaulay modules have been a central object of study in commutative algebra, following the foundational work of Cohen and Macaulay \cite{Cohen1946, Macaulay1994}. This concept was later broadened by Stanley and Schenzel, who independently defined sequentially Cohen--Macaulay modules \cite{Stanley1996, Schenzel1999}. While these generalizations are powerful, they still exclude many modules and rings that possess significant algebraic properties.

This paper introduces a new, broader class of modules, which we term \emph{initially Cohen--Macaulay}. Our approach is to replace the classical Krull dimension in the definition of Cohen-Macaulay modules with a different invariant, which we call the \emph{initial dimension}, given by
$$\mdim N=\inf\{\coht P\mid P\in\ass_RN\}.$$
A finitely generated module $N$ over a local ring $(R,M)$ is then defined as \emph{initially Cohen--Macaulay} if $N=0$ or if its depth equals this new invariant:
$$\dep N=\mdim N.$$

This definition immediately includes sequentially Cohen--Macaulay modules and offers a deeper combinatorial and homological perspective on these modules and their associated invariants. The well-known fact that depth is bounded by the coheight of any associated prime \cite[Theorem 17.2]{Matsumura1989} provides the foundational inequality for this new property. We state this formally:

\begin{theorem}[Proposition \ref{depth-mdim}; c.f. {\cite[Theorem 17.2]{Matsumura1989}}]
	For any finitely generated $R$-module, we have
	$$\dep N\le\mdim N.$$
	Equality holds if and only if $N$ is initially Cohen--Macaulay.
\end{theorem}

This theorem establishes that the initially Cohen--Macaulay condition is a natural weakening of the classical Cohen--Macaulay property. It serves as the starting point for investigating the structural, homological, and combinatorial consequences of this definition.

This paper has two primary goals. We first build the foundational theory of initially Cohen--Macaulay modules, establishing their relationship to both the classical Cohen--Macaulay and sequentially Cohen--Macaulay classes. Our second goal is to demonstrate the utility of this new definition by applying it to modules arising in combinatorics, specifically Stanley--Reisner rings and graph edge ideals. In doing so, we gain new insights into their homological invariants. We begin by establishing a key structural theorem:

\begin{theorem}[Proposition \ref{ICM-unmixed}]
	A finitely generated $R$-module $N$ is Cohen--Macaulay if and only if it is initially Cohen--Macaulay and counmixed, meaning all associated primes of $N$ have the same coheight.
\end{theorem}

This theorem shows that the initially Cohen--Macaulay property is precisely the Cohen--Macaulay property without the counmixed condition. Furthermore, we establish that the property is well-behaved; in Propositions \ref{regular element} and \ref{localization initial dimension}, we prove that it is preserved under localization and when quotienting by regular elements.

\quad One of the key tools in our analysis is the dimension filtration introduced by Schenzel \cite{Schenzel1999}. For an $R$-module $N$ with $\dim N = d$, this filtration is given by
$$0=N_0\subset N_{1}\subset\dots\subset N_{d}=N,$$
where $N_i$ is the largest submodule of $N$ with $\dim N_i\le i$. This filtration is used to define sequentially Cohen--Macaulay modules. According to Schenzel \cite[Definition 4.1]{Schenzel1999} and Stanley \cite[Section III, 2.9]{Stanley1996} (for graded modules), an $R$-module is sequentially Cohen--Macaulay if every quotient $N_i/N_{i-1}$ of this filtration is Cohen--Macaulay (starting from $i = \mdim N$, see Lemma \ref{filtration-indim}). We connect our new definition to this concept, showing that the initially Cohen--Macaulay property depends only on the \emph{first} non-zero piece of this filtration.

\begin{theorem}[Proposition \ref{ICM-filtration}]
	A finitely generated $R$-module $N$ is initially Cohen--Macaulay if and only if $N_{\mdim N}$ is Cohen--Macaulay in the dimension filtration of $N$.
\end{theorem}

This result immediately shows that every sequentially Cohen--Macaulay module is also initially Cohen--Macaulay. Consequently, many classical properties and results known for Cohen--Macaulay or sequentially Cohen--Macaulay modules extend naturally to this broader class. We also provide a purely dimension-theoretic characterization:

\begin{theorem}[Proposition \ref{prop:ICM}]
	Let $(R,M)$ be a local ring and $N$ a finitely generated $R$-module with $\mdim N=b$. Then $N$ is initially Cohen--Macaulay if and only if there exists a sequence $a_1,\dots,a_b\in M$ such that $N_b/(a_1,\dots,a_i)N_b$
	is counmixed of dimension $b-i$ for all $1\le i\le b$, where $N_b$ is the largest submodule of $N$ of dimension $b$.
\end{theorem}

\quad The initially Cohen--Macaulay theory extends naturally to the combinatorial setting of Stanley--Reisner rings. We call a simplicial complex $\Delta$ initially Cohen--Macaulay if its Stanley--Reisner ring $\K[\Delta]$ is initially Cohen--Macaulay. Our main results in this area generalize three classical theorems. Using the concept of a degree resolution \cite[Definition 3.1]{Namiq2025}, we generalize the Eagon--Reiner theorem \cite[Theorem 3]{EagonReiner1998}:

\begin{theorem}[Proposition \ref{PCM degree resolution}]
	A simplicial complex $\Delta$ is initially Cohen--Macaulay if and only if its Alexander dual ideal has a degree resolution.
\end{theorem}

This leads to a simple and powerful characterization in terms of the complex's skeletons:

\begin{theorem}[Proposition \ref{PCM CM}]
	A simplicial complex $\Delta$ is initially Cohen--Macaulay if and only if its initial dimension skeleton $\Delta^{\mdim\Delta}$ is Cohen--Macaulay.
\end{theorem}

This result refines \cite[Theorem 3.3]{Duval1996} and provides a simpler computational check. We also refine the characterization of sequentially Cohen-Macaulay complexes:

\begin{theorem}[Proposition \ref{SCM}]
	A simplicial complex $\Delta$ is sequentially Cohen--Macaulay if and only if $\Delta^{[i]}$ is Cohen--Macaulay for all $\mdim\Delta\leq i\leq\dim\Delta$.
\end{theorem}

Furthermore, we generalize Reisner’s criterion \cite[Theorem 1]{Reisner1976} from Cohen--Macaulay to initially Cohen--Macaulay complexes:

\begin{theorem}[Proposition \ref{pro 1}]
	A simplicial complex $\Delta$ is initially Cohen--Macaulay if and only if for every face $F$ in $\Delta$, the reduced homology $\widetilde{H}_i(\link_\Delta(F);\K)=0$ for all $i<\mdim\Delta-|F|$.
\end{theorem}

\quad  In the final sections, we apply these results to graph theory. We show that all path graphs and, more broadly, all chordal graphs are initially Cohen--Macaulay (Proposition \ref{PCM Pn}). We also find that:
\begin{theorem}[Proposition \ref{PCM Cn}]
	Cycle graphs $C_n$ are initially Cohen--Macaulay except when $n\equiv1\pmod{3}$.
\end{theorem}
We provide criteria for co-chordal graphs:
\begin{theorem}[Proposition \ref{PCM co-chordal}]
	The independence complex of a co-chordal graph is initially Cohen--Macaulay if and only if it is weakly connected.
\end{theorem}
As a corollary, we show that the independence complex of a co-chordal graph is sequentially Cohen--Macaulay if and only if it is stably connected (Corollary \ref{SCM stably connected}). Furthermore, we define and classify bi-initially Cohen--Macaulay graphs. These are graphs whose independence complexes are initially Cohen--Macaulay and whose Stanley--Reisner ideals admit a degree resolution (Proposition \ref{bi-PCM graphs}). We conclude by classifying all graphs for which the projective dimension equals the maximum vertex degree, a problem previously studied only in special cases.

\begin{theorem}[Proposition \ref{pdim max}]
	For a graph $G$, the projective dimension of $\K[\Delta_G]$ equals its maximum vertex degree if and only if $\Delta_G$ is initially Cohen--Macaulay and has a minimum facet containing a free vertex.
\end{theorem}

\quad This article is structured as follows. We first review essential background concepts in Section \ref{sec2}. The core algebraic theory of initially Cohen--Macaulay modules, including their behavior with regular sequences and filtrations, is developed in Section \ref{sec3}. In Section \ref{sec5}, we apply this theory to Stanley--Reisner rings, establishing several combinatorial characterizations. Section \ref{sec6} investigates applications to graph theory, and concludes with our classification of graphs where projective dimension equals maximum vertex degree.

\subsection*{Acknowledgement}
The author gratefully thanks Dr. Chwas Ahmed for his fruitful discussions on this work and for his generous willingness to review the manuscript in detail. His careful reading and thoughtful comments have significantly strengthened this paper.

\section{Preliminaries}\label{sec2}
\quad We recall fundamental concepts from commutative algebra, simplicial complexes, and graph theory. For a comprehensive background, we refer the reader to \cite{BrunsHerzog1998,Matsumura1989,Stanley1996,Villarreal2015} and the references therein. All rings $R$ are assumed to be commutative and Noetherian, and $N$ denotes an $R$-module.

\subsection{Commutative Algebra}
\quad We first recall standard concepts from commutative algebra. The \emph{height} of a prime ideal $P$, $\hte P$, is the length $h$ of the longest chain of prime ideals $P_0 \subset \dots \subset P_h = P$. The \emph{dimension} of the ring, $\dim R$, is the supremum of all such heights. Dually, the \emph{coheight}, $\coht P = \dim R/P$, is the supremum of lengths of chains starting at $P$. These are related by $\hte P=\dim R_P$ and $\hte P+\coht P\le\dim R$.

For an $R$-module $N$, its \emph{dimension} is $\dim N = \dim R/\ann_RN$, where $\ann_RN$ is the annihilator of $N$, $\ann_RN=\{r\in R\mid rn=0\text{ for all }n\in N\}$. A prime ideal $P$ is \emph{associated} to $N$ if $P=\ann_R(n)$ for some $n\in N$. The set of \emph{associated primes} of $N$ is denoted $\ass_RN$. From this set, we define the \emph{height} $\hte N = \inf\{\hte P \mid P \in\ass_RN\}$ and \emph{big height} $\bight N = \sup\{\hte P \mid P\in\ass_RN\}$. An $R$-module $N$ is \emph{unmixed} if all its associated primes have the same height.

An element $r\in R$ is \emph{$N$-regular} if it is not a zerodivisor on $N$. A sequence $r_1,\dots,r_d \in I$ is a maximal \emph{$N$-regular sequence in $I$} if $r_i$ is $N/(r_1,\dots,r_{i-1})N$-regular for all $i$ and the sequence cannot be extended. The length of such a sequence is the \emph{$I$-depth} of $N$, $\dep_IN$. If $(R,M)$ is a local ring, we write $\dep N = \dep_M N$.

For a standard graded ring $R$ and a finitely generated graded $R$-module $N$, the minimal graded free resolution
$$\cdots\longrightarrow\bigoplus_jR(-j)^{\beta_{i,j}}\longrightarrow\cdots\longrightarrow\bigoplus_jR(-j)^{\beta_{0,j}}\longrightarrow N\longrightarrow 0$$
gives the graded Betti numbers $\beta_{i,j}(N)$. From these, we define the \emph{projective dimension} $\pdim N=\max\{i\mid\beta_{i,j}(N)\ne0\}$, the \emph{Castelnuovo--Mumford regularity} $\reg N=\max\{j-i\mid\beta_{i,j}(N)\ne0\}$, and the \emph{maximum degree} $\deg N=\max\{j\mid\beta_{0,j}(N)\ne0\}$. The resolution is an \emph{$r$-linear resolution} if $\beta_{i,j}(N) \ne 0$ only for $j=i+r$, and a \emph{degree resolution} if $\reg N=\deg N$.

\subsection{Simplicial Complexes}
\quad Let $X=\{x_1,\dots,x_n\}$ be a finite vertex set. A \emph{simplicial complex} $\Delta$ on $X$ is a collection of subsets of $X$, called \emph{faces}, such that $\Delta$ is closed under taking subsets. The \emph{dimension} of a face $F$ is $\dim F=|F|-1$, and the \emph{dimension of $\Delta$} is $\dim\Delta=\max\{\dim F\mid F\in\Delta\}$. A maximal face with respect to inclusion is a \emph{facet}. We denote the set of facets by $\f(\Delta)$. A complex is \emph{pure} if all its facets share the same dimension. A complex $\Delta$ is \emph{connected} if any two facets can be joined by a path of facets that overlap. A subset of $X$ not in $\Delta$ is a \emph{non-face}, and $\n(\Delta)$ denotes the set of minimal non-faces.

The \emph{$i$-skeleton} $\Delta^i$ consists of all faces with dimension $\le i$, while the \emph{pure $i$-skeleton} $\Delta^{[i]}$ is the complex generated by all $i$-dimensional faces.

For a face $F\in\Delta$, its \emph{deletion} is $\del_\Delta(F)=\{F'\in\Delta\mid F'\cap F=\emptyset\}$ and its \emph{link} is $\link_\Delta(F)=\{F'\in\Delta\mid F'\cap F=\emptyset, F'\cup F\in\Delta\}$. The \emph{Alexander dual} of $\Delta$ is $\Delta^\vee=\{ X\setminus F\mid F\notin\Delta\}$.

Let $R=\K[x\mid x\in X]$ be a polynomial ring over a field $\K$. The \emph{Stanley--Reisner ideal} $I_\Delta$ is the squarefree monomial ideal generated by the minimal non-faces:
$$I_\Delta=\left(\prod_{x_i\in F}x_i\mid F\in\n(\Delta)\right).$$
The \emph{Stanley--Reisner ring} is the quotient $\K[\Delta]=R/I_\Delta$. The Alexander dual ideal $I^\vee$ is the Stanley--Reisner ideal $I_{\Delta^\vee}$ of the dual complex.

The \emph{$f$-vector} $f(\Delta)=(f_{-1},f_0,\dots,f_{d-1})$ records the number of faces in each dimension $i$. The \emph{$h$-vector} $(h_0, \dots, h_d)$ and the \emph{$h$-polynomial} $h_{\K[\Delta]}(t)=\sum h_i t^i$ are defined by the relation $\sum_{i=0}^df_{i-1}t^i(1-t)^{d-i}=\sum_{i=0}^dh_it^i$.
The \emph{reduced simplicial homology} $\widetilde{H}_q(\Delta;\K)$ measures the $q$-dimensional topological holes of $\Delta$; in particular, $\widetilde{H}_0(\Delta;\K)=0$ measures connectivity.

\subsection{Graphs and Hypergraphs}
\quad A \emph{hypergraph} $\mathcal{H}$ on a vertex set $X$ is a collection of subsets of $X$, called edges, such that no edge properly contains another. A \emph{graph} $G$ is a hypergraph where all edges have size 2. The \emph{edge ideal} $I(\mathcal{H})$ is the squarefree monomial ideal generated by the edges:
$$I(\mathcal{H})=\left(\prod_{x_i\in e}x_i\mid e\in\mathcal{H}\right).$$
A subset $F\subseteq X$ is an \emph{independent set} of $\mathcal{H}$ if $F$ contains no edge $e\in E(\mathcal{H})$. The \emph{independence complex} $\Delta_\mathcal{H}$ is the simplicial complex formed by all independent sets of $\mathcal{H}$.

A \emph{vertex cover} is a subset $C\subseteq X$ that intersects every edge of $\mathcal{H}$. A \emph{clique} in a graph $G$ is a set of vertices that are all pairwise adjacent. The \emph{clique complex} $\Delta(G)$ is the simplicial complex formed by all cliques of $G$.

We use the following standard graph families:
\begin{itemize}
	\item \emph{Path graph} $P_n$: Edges $\{x_i, x_{i+1}\}$ for $1 \le i < n$.
	\item \emph{Cycle graph} $C_n$: The graph $P_n$ plus the edge $\{x_1, x_n\}$.
	\item \emph{Chordal graph}: A graph containing no induced cycles $C_k$ for $k\ge4$.
	\item \emph{Co-chordal graph}: A graph whose complement $\overline{G}$ is chordal.
\end{itemize}

%%%%%%%%%%%%%%%%%%%%%%%%%%%%%%%%%%%%%%%%%%%%%%%%%%%%%%%%%%%%%%%%%%

\section{Initial Dimension and Initially Cohen--Macaulay Modules}
\label{sec3}

\quad This section establishes the foundational theory of initially Cohen--Macaulay modules. We begin by defining the \emph{initial dimension} of an $R$-module, an invariant based on the coheights of its associated primes. We then use this to define the initially Cohen--Macaulay property and explore its primary characterizations, connecting it to classical module properties. Throughout, $R$ remains a commutative Noetherian ring and $N$ an $R$-module.

\begin{definition}
	The \emph{initial dimension} of $N$, denoted $\mathrm{indim}\ N$, is the infimum of the coheights of its associated primes:
	$$\mathrm{indim}\ N=\inf\{\mathrm{coht}\ P\mid P\in\mathrm{Ass}_RN\}.$$
\end{definition}

\begin{remark}
	\begin{enumerate}
		\item This definition contrasts with the Krull dimension. Since $\coht P=\dim R/P$, we have:
		$$\mathrm{indim}\ N=\inf\{\dim R/P\mid P\in\mathrm{Ass}_RN\},\quad\dim N=\sup\{\dim R/P\mid P\in\mathrm{Ass}_RN\}.$$
		\item For any $P\in\mathrm{Ass}_RN$, the following inequalities hold, as $\mathrm{ht}\ P\ge0$:
		$$\mathrm{indim}\ R\le\mathrm{coht}\ P \le \mathrm{ht}\ P+\mathrm{coht}\ P\le\dim R.$$
		\item The following bounds are immediate:
		$$\mathrm{indim}\ R\le\mathrm{bight}\ N+\mathrm{indim}\ N\le\dim R\ \text{ and }\ \mathrm{indim}\ R\le\mathrm{ht}\ N+\dim N\le\dim R.$$
	\end{enumerate}
\end{remark}

\begin{example}
	Let $R=\K[x,y]$ and $I=(x^2,xy)$. The associated primes are $\mathrm{Ass}_RR/I=\{(x),(x,y)\}$. Their coheights are $\mathrm{coht}(x)=1$ and $\mathrm{coht}(x,y)=0$. Therefore, $\mathrm{indim}\ R/I=0$ while $\dim R/I=1$.
\end{example}

\begin{lemma}
	\label{counmixed-implies-minimal}
	Let $P$ be an associated prime of $N$ such that $\mathrm{ht}\ P+\mathrm{coht}\ P=\mathrm{indim}\ N$. Then $P$ is a minimal prime.
\end{lemma}

\begin{proof}
	The hypothesis is $\mathrm{ht}\ P+\mathrm{coht}\ P=\mathrm{indim}\ N$. By definition, $\mathrm{indim}\ N \le \mathrm{coht}\ P$. Substituting this into the hypothesis gives $\mathrm{ht}\ P + \mathrm{coht}\ P \le \mathrm{coht}\ P$, which implies $\mathrm{ht}\ P \le 0$. Since height is non-negative, so $\mathrm{ht}\ P = 0$, and thus $P$ is minimal.
\end{proof}

\begin{definition}
	\label{counmixed}
	An $R$-module $N$ is \emph{counmixed} if all its associated primes have the same coheight (equivalently, the same dimension).
\end{definition}

\begin{lemma}
	\label{counmixed-equiv}
	An $R$-module $N$ is counmixed if and only if $\mathrm{indim}\ N=\dim N$.
\end{lemma}

\begin{proof}
	This is a direct consequence of the definitions, as $\mathrm{indim}\ N$ is the infimum of the set $\{\dim R/P\mid P\in\mathrm{Ass}_RN\}$ and $\dim N$ is its supremum. These values are equal if and only if the set contains a single value.
\end{proof}

\begin{corollary}
	\label{domain-counmixed}
	Every integral domain is counmixed.
\end{corollary}

\begin{proof}
	As $R$ is a domain, $\mathrm{Ass}\ R = \{(0)\}$. Thus $\mathrm{indim}\ R = \mathrm{coht}(0) = \dim R$. The result follows from Lemma \ref{counmixed-equiv}.
\end{proof}

\begin{corollary}
	If $R=\K[x_1,\dots,x_n]$ is a polynomial ring over a field $\K$, it is counmixed, and $\mathrm{indim}\ R=n$.
\end{corollary}

\begin{proof}
	The ring $R$ is an integral domain of dimension $n$, so Corollary \ref{domain-counmixed} applies.
\end{proof}

\begin{lemma}
	\label{counmixed-ring}
	If $R$ is a counmixed ring, then it is unmixed. Furthermore, for every prime $P\subset R$, we have
	$$\mathrm{ht}\ P+\mathrm{coht}\ P=\dim R.$$
\end{lemma}

\begin{proof}
	By Lemma \ref{counmixed-equiv}, the hypothesis implies $\mathrm{indim}\ R = \dim R$. For any $P \in \mathrm{Ass}\ R$, the inequality $\dim R = \mathrm{indim}\ R \le \mathrm{ht}\ P + \mathrm{coht}\ P \le \dim R$ forces equality. Lemma \ref{counmixed-implies-minimal} then shows $\mathrm{ht}\ P = 0$ for all $P \in \mathrm{Ass}\ R$, so $R$ is unmixed. The second assertion follows from the fact every primes ideal $Q$ of $R$ contains a minimal (associated) prime $P$ with $\mathrm{ht}\ P = 0$ and $\mathrm{coht}\ P = \dim R$, giving $\mathrm{ht}\ Q+\mathrm{coht}\ Q=\mathrm{coht}\ P=\dim R$.
\end{proof}

\begin{corollary}
	\label{ht-bight}
	If $R$ is counmixed and $N$ is an $R$-module, then:
	$$\mathrm{bight}\ N+\mathrm{indim}\ N=\dim R\quad\text{and}\quad\mathrm{ht}\ N+\dim N=\dim R.$$
\end{corollary}
\qed

\begin{example}
	The converse of Lemma \ref{counmixed-ring} is false; unmixedness does not imply counmixedness. Let $S=\K[x,y,z]/(xy,xz)$. The associated primes are $\mathrm{Ass}\ S=\{(\bar{x}),(\bar{y},\bar{z})\}$. Both have height zero, so $S$ is unmixed. However, their coheights differ, as $\mathrm{coht}(\bar{x})=2$ and $\mathrm{coht}(\bar{y},\bar{z})=1$. Thus, $S$ is not counmixed.
\end{example}

We now introduce the central definition of this paper. We restrict our setting to a commutative Noetherian local ring $(R,M)$.

\begin{definition}
	Let $(R,M)$ be a local ring. A finitely generated $R$-module $N$ is \emph{initially Cohen--Macaulay} if $N=0$ or
	$$\mathrm{depth}\ N=\mathrm{indim}\ N.$$
	If $R$ itself satisfies this as an $R$-module, $R$ is an \emph{initially Cohen--Macaulay ring}.
\end{definition}

\begin{proposition}\label{depth-mdim}
	Let $(R,M)$ be a local ring and $N$ a finitely generated $R$-module. Then
	$$\mathrm{depth}\ N\le\mathrm{indim}\ N.$$
	Equality holds if and only if $N$ is initially Cohen--Macaulay.
\end{proposition}

\begin{proof}
	The depth of $N$ is bounded by $\dim R/P = \mathrm{coht}\ P$ for every $P \in \mathrm{Ass}_R N$ \cite[Theorem 17.2]{Matsumura1989}. Taking the infimum over all $P \in \mathrm{Ass}_R N$ gives $\mathrm{depth}\ N \le \mathrm{indim}\ N$. The module $N$ is initially Cohen--Macaulay precisely when this inequality is an equality.
\end{proof}

\begin{corollary}\label{maximal-ICM}
	Let $(R,M)$ be a local ring and $N$ a finitely generated $R$-module. Then $M\in\mathrm{Ass}_RN$ if and only if $\mathrm{indim}\ N=0$. In particular, any module with $\mathrm{indim}\ N=0$ is initially Cohen--Macaulay.
\end{corollary}

\begin{proof}
	If $M\in\mathrm{Ass}_R N$, then $\mathrm{indim}\ N \le \mathrm{coht}\ M = 0$. Conversely, if $\mathrm{indim}\ N = 0$, there exists $P \in \mathrm{Ass}_R N$ with $\mathrm{coht}\ P = 0$, which implies $P=M$. Such modules are initially Cohen--Macaulay because $0 \le \mathrm{depth}\ N \le \mathrm{indim}\ N = 0$.
\end{proof}

\begin{corollary}
	Let $N$ be a finitely generated $R$-module and $P\in\mathrm{Ass}_R N$. Then the localization $N_P$ is an initially Cohen--Macaulay $R_P$-module, and $\mathrm{indim}\ N_P=0$.
\end{corollary}
\qed

The next result connects the initially Cohen--Macaulay property to the classical Cohen--Macaulay condition.

\begin{proposition}\label{ICM-unmixed}
	Let $(R,M)$ be a local ring and $N$ a finitely generated $R$-module. Then $N$ is Cohen--Macaulay if and only if $N$ is initially Cohen--Macaulay and counmixed.
\end{proposition}

\begin{proof}
	A module $N$ is Cohen--Macaulay if $\mathrm{depth}\ N = \dim N$. By Lemma \ref{counmixed-equiv}, $N$ is counmixed if and only if $\mathrm{indim}\ N = \dim N$. The proposition follows, as
	$N$ is Cohen--Macaulay if and only if $\mathrm{depth}\ N = \dim N$
	 if and only if $\mathrm{depth}\ N = \mathrm{indim}\ N$ and $\mathrm{indim}\ N = \dim N$ if and only if $N$ is initially Cohen--Macaulay and counmixed.
\end{proof}

\begin{corollary}
	Every polynomial ring $R = \K[x_1,\dots,x_n]$ over a field $\K$ is initially Cohen--Macaulay.
\end{corollary}
\qed

\begin{proposition}\label{pdim-bight}
	Let $R$ be a Cohen--Macaulay local ring and $N$ a finitely generated $R$-module. Then
	$$\mathrm{pdim}\ N \ge\mathrm{bight}\ N.$$
	Equality holds if and only if $N$ is initially Cohen--Macaulay.
\end{proposition}

\begin{proof}
	The Auslander--Buchsbaum formula \cite[Theorem 3.1]{AuslanderBuchsbaum1957} states $\mathrm{pdim}\ N+\mathrm{depth}\ N=\mathrm{depth}\ R$. Since $R$ is Cohen--Macaulay, $\mathrm{depth}\ R=\dim R$. As a Cohen--Macaulay ring, $R$ is also counmixed, so Corollary \ref{ht-bight} gives $\dim R = \mathrm{bight}\ N+\mathrm{indim}\ N$. Equating these yields $\mathrm{pdim}\ N+\mathrm{depth}\ N = \mathrm{bight}\ N+\mathrm{indim}\ N$.
	From Proposition \ref{depth-mdim}, we know $\mathrm{depth}\ N \le \mathrm{indim}\ N$. This implies $\mathrm{pdim}\ N \ge \mathrm{bight}\ N$. Equality holds in this new inequality if and only if $\mathrm{depth}\ N = \mathrm{indim}\ N$, which is the definition of $N$ being initially Cohen--Macaulay.
\end{proof}

\begin{corollary}\label{pdim-hte}
	Let $R$ be a Cohen--Macaulay local ring and $N$ a finitely generated $R$-module. Then $N$ is Cohen--Macaulay if and only if $\mathrm{pdim}\ N=\mathrm{ht}\ N$.
\end{corollary}

\begin{proof}
	This follows by combining Proposition \ref{ICM-unmixed} (noting $N$ is Cohen--Macaulay if and only if $\mathrm{bight}\ N = \mathrm{ht}\ N$ and $N$ is initially Cohen--Macaulay), Proposition \ref{pdim-bight}, and Corollary \ref{ht-bight}.
\end{proof}

We now analyze how the initial dimension and the initially Cohen--Macaulay property behave with respect to regular sequences. The following lemma is standard:

\begin{lemma}\label{8}
	Let $(R,M)$ be a local ring, $N$ a finitely generated $R$-module, and $a\in M$ an $N$-regular element. If $P\in\mathrm{Ass}_RN$, any prime $P'$ minimal over $(P,a)$ belongs to $\mathrm{Ass}_RN/aN$. In particular,
	$$\mathrm{indim}\ N/aN\le\mathrm{indim}\ N-1\quad\text{and}\quad\dim N/aN=\dim N-1.$$
\end{lemma}

This result can be extended inductively to regular sequences of arbitrary length.

\begin{corollary}\label{regular indim}
	Let $(R,M)$ be a local ring and $N$ a finitely generated $R$-module. If $a_1,\dots,a_r\in M$ form an $N$-regular sequence, then
	$$\mathrm{indim}\ N/(a_1,\dots,a_r)N\le\mathrm{indim}\ N-r\quad\text{and}\quad\dim N/(a_1,\dots,a_r)N=\dim N-r.$$
\end{corollary}
\qed

\begin{proposition}\label{regular element}
	Let $(R,M)$ be a local ring, $N$ a finitely generated $R$-module, and $a=a_1,\dots,a_r\in M$ an $N$-regular sequence. Then $N$ is initially Cohen--Macaulay if and only if $N/aN$ is initially Cohen--Macaulay and $\mathrm{indim}\ N/aN=\mathrm{indim}\ N-r$.
\end{proposition}

\begin{proof}
	Assume $N$ is initially Cohen--Macaulay. Then $\mathrm{depth}\ N/aN = \mathrm{depth}\ N - r = \mathrm{indim}\ N - r$.
	From Corollary \ref{regular indim}, we have $\mathrm{indim}\ N/aN \le \mathrm{indim}\ N - r$.
	Combining these, $\mathrm{depth}\ N/aN \ge \mathrm{indim}\ N/aN$. Proposition \ref{depth-mdim} provides the reverse inequality, so we must have $\mathrm{depth}\ N/aN = \mathrm{indim}\ N/aN$. This shows $N/aN$ is initially Cohen--Macaulay and forces $\mathrm{indim}\ N/aN = \mathrm{indim}\ N - r$.
	
	Conversely, assume $N/aN$ is initially Cohen--Macaulay and $\mathrm{indim}\ N/aN = \mathrm{indim}\ N - r$. Then $\mathrm{depth}\ N - r = \mathrm{depth}\ N/aN = \mathrm{indim}\ N/aN = \mathrm{indim}\ N - r$. This implies $\mathrm{depth}\ N = \mathrm{indim}\ N$, so $N$ is initially Cohen--Macaulay.
\end{proof}

The initially Cohen--Macaulay property is also well-behaved under localization.

\begin{proposition}\label{localization initial dimension}
	Let $(R,M)$ be a local ring and $N$ a finitely generated initially Cohen--Macaulay $R$-module. For any prime ideal $P$ of $R$, the localization $N_P$ is an initially Cohen--Macaulay $R_P$-module. Moreover, if $N_P\ne0$, then $\mathrm{depth}_PN=\mathrm{depth}\ N_P = \mathrm{indim}\ N_P$.
\end{proposition}

\begin{proof}
	Assume $N_P\ne0$. The inequalities $\mathrm{depth}_PN\le\mathrm{depth}\ N_{P}\le\mathrm{indim}\ N_P$ are standard. We prove $\mathrm{depth}_PN=\mathrm{indim}\ N_P$ by induction on $t = \mathrm{depth}_PN$.
	
	If $t=0$, then $\mathrm{depth}\ N_P = 0$. This means the maximal ideal $PR_P$ of $R_P$ is an associated prime of $N_P$. By Corollary \ref{maximal-ICM}, this implies $\mathrm{indim}\ N_P = 0$. Thus, $\mathrm{depth}_PN = \mathrm{indim}\ N_P = 0$.
	
	Now assume $t > 0$ and the claim holds for all initially Cohen--Macaulay modules with $P$-depth less than $t$. Choose an $N$-regular element $a\in P$ and set $N'=N/aN$. By Proposition \ref{regular element}, $N'$ is also initially Cohen--Macaulay and $\mathrm{indim}\ N' = \mathrm{indim}\ N - 1$. Since localization preserves exactness, the element $a$ remains $N_P$-regular in $R_P$. By localizing at $P$, we obtain
	$\mathrm{indim}\ (N_P/aN_P)=\mathrm{indim}\ N'_P=\mathrm{indim}\ N_P - 1$. We have $\mathrm{depth}_PN' = \mathrm{depth}_PN - 1 = t-1$. By the induction hypothesis applied to $N'$, we get $N'_P$ is initially Cohen--Macaulay and $\mathrm{depth}_PN'= \mathrm{indim}\ N'_P$. We can now chain the equalities:
	$$\mathrm{depth}_PN - 1 = \mathrm{depth}_PN' = \mathrm{indim}\ N'_P = \mathrm{indim}\ (N_P/aN_P)=\mathrm{indim}\ N_P - 1.$$
	Therefore, $\mathrm{depth}_P N = \mathrm{indim}\ N_P$, and $N_P$ is initially Cohen--Macaulay.
\end{proof}

We now connect the initially Cohen--Macaulay property to Schenzel's dimension filtration \cite{Schenzel1999}.

\begin{lemma}\label{filtration-indim}
	Let $(R,M)$ be a local ring and $N$ a finitely generated $R$-module. Let $N_i$ be the largest submodule of $N$ with $\dim N_i\le i$, and let $b=\mathrm{indim}\ N$, $d=\dim N$. The dimension filtration of $N$ starts at index $b$:
	$$0=N_{b-1}\subset N_b\subset N_{b+1}\subset\dots\subset N_d=N.$$
	Furthermore, $\mathrm{indim}\ N=\dim N_b=b$.
\end{lemma}

\begin{proof}
	Schenzel's filtration \cite{Schenzel1999} ensures that $\mathrm{Ass}(N_i/N_{i-1}) = \{P\in\mathrm{Ass}\ N \mid \dim R/P = i\}$. Since $b = \mathrm{indim}\ N$ is the minimum dimension of $R/P$ for $P \in \mathrm{Ass}\ N$, the set $\mathrm{Ass}(N_i/N_{i-1})$ is empty for all $i < b$. This means $N_i = 0$ for $i < b$, so $N_b$ is the first non-zero module in the filtration. It follows that $\dim N_b = \dim (N_b / N_{b-1}) = b = \mathrm{indim}\ N$.
\end{proof}

The following result shows that the depth of $N$ is controlled entirely by the first non-zero piece of its dimension filtration.

\begin{proposition}\label{dep N=dep Nb}
	Let $(R,M)$ be a local ring and $N$ a nonzero finitely generated $R$-module with dimension filtration
	$$0\subset N_b\subset N_{b+1}\subset\cdots\subset N_d=N,$$
	where $b=\mdim N$ and $d=\dim N$. Then $\dep N=\dep N_i$ for each $b\le i\le d$.
\end{proposition}

\begin{proof}
	We first show by induction on $r$ that if $a_1,\dots,a_r$ is an $N$-regular sequence, then it is also an $N_i$-regular sequence for all $b\le i\le d$.
	
	For $r=1$, suppose $a_1$ is $N$-regular. Since $N_i\subseteq N$, multiplication by $a_1$ on each $N_i$ is injective, hence $a_1$ is $N_i$-regular. By the second isomorphism theorem, we have $N_i\cap a_1N=a_1N_i$, which yields
	$$\frac{N_i+a_1N}{a_1N}\cong\frac{N_i}{N_i\cap a_1N}=\frac{N_i}{a_1N_i}.$$
	Thus, up to isomorphism,
	$$\frac{N_i}{a_1N_i}\subseteq\frac{N_{i+1}}{a_1N_{i+1}}\quad\text{for }b\le i\le d-1.$$
	Furthermore, Lemma \ref{8} implies $\dim N_i/a_1N_i=\dim N_i-1 = i-1$ for each $b\le i\le d$. Hence the chain
	\begin{equation}\label{eq1}
		0\subset\frac{N_b+a_1N}{a_1N}\subset\frac{N_{b+1}+a_1N}{a_1N}\subset\cdots\subset\frac{N_d+a_1N}{a_1N}=\frac{N}{a_1N}
		\end{equation}
		is a part of the dimension filtration of $N/a_1N$, which starts at dimension $b-1$.
		
		For $r>1$, suppose the result holds for sequences of length $r-1$. Let $a_1,\dots,a_r$ be an $N$-regular sequence. Then $a_2,\dots,a_r$ is an $(N/a_1N)$-regular sequence. By the inductive hypothesis applied to the filtration \eqref{eq1}, $a_2,\dots,a_r$ is $(N_i/a_1N_i)$-regular for all $i \ge b$. Since $a_1$ is a nonzerodivisor on every $N_i$, it follows that $a_1,\dots,a_r$ is an $N_i$-regular sequence for all $i$. As $r$ can be any integer up to $\dep N$, this implies
		$$\dep N\le\dep N_i\quad\text{for all }b\le i\le d.$$
		We now prove $\dep N = \dep N_b$ by induction on $b=\mdim N$.
		
		If $b=0$, then $\dim N_b=0$, so $\dep N_b=0$. Since $\mdim N = b = 0$, $M\in\ass_RN$, which implies $\dep N=0$. Thus $\dep N=\dep N_b=0$. As $\dep N \le \dep N_i$, the equality $\dep N = \dep N_i = 0$ holds for all $i$, and the assertion holds.
		
		Now suppose $b\ge1$ and the statement holds for all modules with initial dimension less than $b$. Since $b \ge 1$, $\mdim N > 0$, which implies $M \notin \ass_R N$. Also, if $\dep N = 0$, then $M \in \ass_R N$, a contradiction. Thus $\dep N > 0$. By prime avoidance, there exists an element
		$$a\in M\setminus\bigcup_{P\in\ass_RN}P,$$
		which is $N$-regular. From the first part of our proof, $a$ is $N_i$-regular for all $i \ge b$, and the induced filtration of $N/aN$ has nonzero layer $(N_b+aN)/aN\cong N_b/aN_b$ which has dimension $b-1$ by \eqref{eq1}. Thus, $\mdim(N/aN) \le b-1$. By the induction hypothesis applied to $N/aN$, we obtain
		$$\dep N/aN=\dep N_i/aN_i\quad\text{for each }b\le i\le d.$$
		Since $a$ is regular on both $N$ and $N_b$, we have $\dep N/aN = \dep N - 1$ and $\dep N_b/aN_b = \dep N_b - 1$. Therefore,
		$$\dep N - 1 = \dep N_b - 1,$$
		which implies $\dep N=\dep N_b$. Combining with $\dep N \le \dep N_i$ for all $i \ge b$, we conclude $$\dep N=\dep N_i\quad\text{ for all }b\le i\le d.$$
\end{proof}

This immediately provides a characterization of initially Cohen--Macaulay modules using the dimension filtration.

\begin{proposition}\label{ICM-filtration}
	Let $(R,M)$ be a local ring and $N$ a finitely generated $R$-module. Then $N$ is initially Cohen--Macaulay if and only if the first non-zero module $N_b$ in its dimension filtration is Cohen--Macaulay (where $b = \mathrm{indim}\ N$).
\end{proposition}

\begin{proof}
	We may assume $N \ne 0$. By Lemma \ref{filtration-indim}, $\mathrm{indim}\ N = \dim N_b = b$.
	Assume $N$ is initially Cohen--Macaulay. By definition, $\mathrm{depth}\ N = \mathrm{indim}\ N = b$. By Proposition \ref{dep N=dep Nb}, $\mathrm{depth}\ N_b = \mathrm{depth}\ N$. Chaining these equalities, $\mathrm{depth}\ N_b = b = \dim N_b$, so $N_b$ is Cohen--Macaulay.
	
	Conversely, assume $N_b$ is Cohen--Macaulay. Then $\mathrm{depth}\ N_b = \dim N_b = b$. By Proposition \ref{dep N=dep Nb}, $\mathrm{depth}\ N = \mathrm{depth}\ N_b$. Thus $\mathrm{depth}\ N = b = \mathrm{indim}\ N$, which means $N$ is initially Cohen--Macaulay.
\end{proof}

\begin{corollary}\label{SCM-implies-ICM}
	Every sequentially Cohen--Macaulay module is initially Cohen--Macaulay.
\end{corollary}

\begin{proof}
	An $R$-module $N$ is sequentially Cohen--Macaulay if every quotient $N_i/N_{i-1}$ of its dimension filtration (for $i \ge b = \mathrm{indim}\ N$) is Cohen--Macaulay \cite[Definition 4.1]{Schenzel1999}, \cite[Section III, 2.9]{Stanley1996}. In particular, the first non-zero factor $N_b/N_{b-1} = N_b$ must be Cohen--Macaulay. By Proposition \ref{ICM-filtration}, this is a sufficient condition for $N$ to be initially Cohen--Macaulay.
\end{proof}

Finally, we provide a characterization in purely dimension-theoretic terms.

\begin{proposition}\label{prop:ICM}
	Let $(R,M)$ be a local ring and $N$ a finitely generated $R$-module with $\mathrm{indim}\ N=b$. Then $N$ is initially Cohen--Macaulay if and only if there exists a sequence $a_1,\dots,a_b\in M$ such that for all $1\le i\le b$, the module $N_b/(a_1,\dots,a_i)N_b$ is counmixed of dimension $b-i$, where $N_b$ is the largest submodule of $N$ of dimension $b$.
\end{proposition}

\begin{proof}
	Suppose $N$ is initially Cohen--Macaulay. By Proposition \ref{ICM-filtration}, $N_b$ is Cohen--Macaulay of dimension $b$. Thus, there exists an $N_b$-regular sequence $a_1,\dots,a_b \in M$. For any $1 \le i \le b$, the quotient $N_b/(a_1,\dots,a_i)N_b$ is also Cohen--Macaulay of dimension $b-i$. By Proposition \ref{ICM-unmixed}, any Cohen--Macaulay module is counmixed. This proves the forward direction.
	
	Conversely, suppose such a sequence $a_1,\dots,a_b$ exists. We claim this sequence is $N_b$-regular. We prove this by induction on $i$. Fix $1 \le i \le b$ and set $N':=N_b/(a_1,\dots,a_{i-1})N_b$. (For $i=1$, $N'=N_b$.) By hypothesis, $N'$ is counmixed of dimension $b-(i-1)$, and $\dim N'/a_iN' = b-i = \dim N' - 1$.
	If $a_i$ were a zerodivisor on $N'$, then $a_i \in P$ for some $P\in\mathrm{Ass}_RN'$. Since $N'$ is counmixed, $\dim R/P = \dim N'$. This would imply $\dim N'/a_iN' \ge \dim R/P = \dim N'$, which contradicts $\dim N'/a_iN' = \dim N' - 1$. Thus, $a_i$ must be $N'$-regular. This completes the claim.
	
	Since $a_1,\dots,a_b$ is an $N_b$-regular sequence, $\mathrm{depth}\ N_b \ge b$. Proposition \ref{depth-mdim} gives the reverse inequality $\mathrm{depth}\ N_b \le \dim N_b = b$. Therefore, $\mathrm{depth}\ N_b = b$. By Proposition \ref{dep N=dep Nb}, $\mathrm{depth}\ N = \mathrm{depth}\ N_b = b = \mathrm{indim}\ N$. Hence, $N$ is initially Cohen--Macaulay.
\end{proof}

\begin{corollary}\label{cor:CM}
	Let $(R,M)$ be a local ring and $N$ a finitely generated $R$-module of dimension $d$. Then $N$ is Cohen-Macaulay if and only if there exists a sequence $a_1,\dots,a_d\in M$ such that $N/(a_1,\dots,a_i)N$ is counmixed of dimension $d-i$ for every $1\le i\le d$.
\end{corollary}
\qed

%%%%%%%%%%%%%%%%%%%%%%%%%%%%%%%%%%%%%%%%%%%%%%%%%%%%%%%%%%%%%%%%%%

\section{Characterizations of Initially Cohen--Macaulay Simplicial Complexes}
\label{sec5}

\quad This section extends the theory by developing specific characterizations for initially Cohen--Macaulay Stanley--Reisner rings. The focus is on criteria related to combinatorial skeletons and reduced homology.

We fix $R=\K[x_1,\dots,x_n]$ as the polynomial ring with the standard grading $\deg(x_i)=1$ for all $i$. Let $I\subset R$ be a squarefree monomial ideal, which is the Stanley--Reisner ideal of a simplicial complex $\Delta$ on the vertex set $X=\{x_1,\dots,x_n\}$, so $I=I_\Delta$.
We define a simplicial complex $\Delta$ as \emph{initially Cohen--Macaulay} over $\K$ if its Stanley--Reisner ring $\K[\Delta]$ is initially Cohen--Macaulay. We may use $\Delta_I$ to emphasize that its Stanley--Reisner ideal is $I$.

Let $M=(x_1,\dots,x_n)$ be the unique graded maximal ideal of $R$. Following \cite{AhmedFrobergNamiq2023}, for any integer $k>0$, $I_k$ denotes the squarefree part of the intersection $I\cap M^k$. This ideal $I_k$ consists of all squarefree monomials in $I$ with degree at least $k$. The following lemma is essential for our subsequent analysis.

\begin{lemma}\label{Alexander dual}
Let $I$ be a squarefree monomial ideal in $R$. For every integer $k\ge\min\{\deg(u)\mid u\in I\}$,
$$(\Delta_{I_k})^\vee=(\Delta_I^\vee)^{n-k-1}.$$
\end{lemma}
\begin{proof}
We proceed by induction on $k$. The base case, where $k=\min\left\{\deg(u)\mid u\in I\right\}$, is trivial. By definition, $I_k$ contains all squarefree monomials in $I$ of degree at least $k$, which implies $I=I_k$. Thus,
$$\left(\Delta_{I_k}\right)^\vee=\Delta_I^\vee=\left(\Delta_I^\vee\right)^{n-k-1},$$
as the Alexander dual $\Delta_I^\vee$ has dimension $n-k-1$.

Next, we prove the result for $q+1$, assuming $q\ge k$. The Alexander dual of $\Delta_{I_{q+1}}$ is, by definition,
$$\left(\Delta_{I_{q+1}}\right)^\vee=\left\langle\overline{e_u}\mid u\in\mathcal{G}(I_{q+1})\right\rangle,$$
where $e_u:=\supp(u)=\{x_i\mid x_i\ \text{ divides }u\}$ and $\overline{e_u}=X\setminus e_u$. The ideal $I_{q+1}$ is generated by
$$I_{q+1}=\left(ux_i\mid\deg(u)=q,\ u\in I_q,\ x_i\notin e_u\right)+\left(v\in I_q\mid\deg(v)>q\right).$$
The Alexander dual therefore becomes
$$\left(\Delta_{I_{q+1}}\right)^\vee=\left\langle\overline{e_{u x_i}}\mid\deg(u)=q,\ u\in I_q,\ x_i\notin e_u\right\rangle\cup\left\langle \overline{e_v} \mid\deg(v)>q,\ v\in I_q\right\rangle.$$
This can be rewritten as
$$(\Delta_{I_{q+1}})^\vee=\lr{e_{\overline{ux_i}}\mid\dim e_{\overline{u}}=n-q-1,e_{\overline{u}}\in\Delta_{I_{q}}^\vee,x_i\in e_{\overline{u}}}\cup\lr{e_{\overline{v}}\mid\dim e_{\overline{v}}<n-q-1,e_{\overline{v}}\in\Delta_{I_{q}}^\vee}.$$
For each $u\in I_q$ with $\deg(u)=q$, the corresponding face $\overline{e_u}\in(\Delta_{I_q})^\vee$ has dimension $n-q-1$. For each $x_i\in\overline{e_u}$, the face $\overline{e_{ux_i}}=\overline{e_u}\setminus\{ x_i \}$ has dimension $n-q-2$. All other faces $\overline{e_v}$ (where $\deg(v)>q$) have dimensions strictly less than $n-q-1$.

Thus, the facets of $\left(\Delta_{I_{q+1}}\right)^\vee$ are generated by removing a single vertex from the $(n-q-1)$-dimensional facets of $\left(\Delta_{I_q}\right)^\vee$. This means
$$(\Delta_{I_{q+1}})^\vee=\left((\Delta_{I_q})^\vee\right)^{n-q-2}.$$
Applying the induction hypothesis, $(\Delta_{I_q})^\vee=(\Delta_I^\vee)^{n-q-1}$, we get
$$(\Delta_{I_{q+1}})^\vee=\left((\Delta_I^\vee)^{n-q-1}\right)^{n-q-2}=(\Delta_I^\vee)^{n-(q+1)-1}.$$
This completes the induction.
\end{proof}

This lemma allows us to extend Theorem 3 of Eagon and Reiner \cite{EagonReiner1998}.
\begin{lemma}\label{EagonReiner}
Let $I$ be a squarefree monomial ideal and $k\ge\min\{\deg(u)\mid u\in I\}$. Then $I_k$ has a linear resolution if and only if its Alexander dual complex, $(\Delta_I^\vee)^{n - k - 1}$, is Cohen--Macaulay.
\end{lemma}
\qed

\begin{remark}\label{9}
From Lemma \ref{EagonReiner} and \cite[Corollary 2.6]{AhmedFrobergNamiq2023}, we can conclude that if $\Delta^i$ is Cohen--Macaulay, then $\Delta^{i-1}$ is also Cohen--Macaulay.
\end{remark}

Following \cite{Namiq2025}, we say a squarefree monomial ideal $I$ has a \emph{degree resolution} if $\reg I=\deg I$, where $\deg I=\max\{\deg(u) \mid u\in\mathcal{G}(I)\}$ and $\mathcal{G}(I)$ is the set of minimal generators of $I$. We now connect the initially Cohen--Macaulay property to this notion.

\begin{proposition}\label{PCM degree resolution}
	A simplicial complex $\Delta$ is initially Cohen--Macaulay if and only if its Alexander dual ideal $I_{\Delta^\vee}$ has a degree resolution.
\end{proposition}
\begin{proof}
	Let $I=I_{\Delta}$. We use the fact that $\bight I=\deg I^\vee$. 
	
	For the forward direction, assume $\Delta$ is initially Cohen--Macaulay. Applying \cite[Theorem 0.2]{Terai1997} along with Proposition \ref{pdim-bight}, we obtain
	$$\reg I^\vee=\pdim\K[\Delta]=\bight I=\deg I^\vee.$$
	This shows that $I^\vee$ has a degree resolution.
	
	Conversely, if $I^\vee$ has a degree resolution, then
	$$\pdim\K[\Delta]=\reg I^\vee=\deg I^\vee=\bight I.$$
	By Proposition~\ref{pdim-bight}, this equality implies that $\Delta$ is initially Cohen--Macaulay.
\end{proof}

The following proposition provides a key characterization of the initially Cohen--Macaulay property in terms of a specific skeleton.

\begin{proposition}\label{PCM CM}
	A simplicial complex $\Delta$ is initially Cohen--Macaulay if and only if its initial-dimension skeleton $\Delta^{\mdim\Delta}$ is Cohen--Macaulay.
\end{proposition}
\begin{proof}
	Let $I=I_{\Delta}$. From Proposition \ref{PCM degree resolution}, $\Delta$ is initially Cohen--Macaulay precisely when $I^\vee$ has a degree resolution. By \cite[Proposition 3.3]{Namiq2025}, this is equivalent to $(I^\vee)_k$ having a linear resolution, where $k=\deg I^\vee$.
	
	We relate $k$ to the initial dimension of $\Delta$:
	$$k=\deg I^\vee=\bight I=n-\mdim\K[\Delta]=n-(\mdim\Delta+1) = n-k-1.$$
	Hence, $\mdim\Delta=n-k-1$.
	
	By Lemma \ref{EagonReiner}, the ideal $(I^\vee)_k$ has a linear resolution if and only if the complex $(\Delta_{I^\vee}^\vee)^{n-k-1}$ is Cohen--Macaulay. Using Lemma \ref{Alexander dual}, this complex is $\Delta^{\mdim \Delta}$. Chaining these equivalences gives the desired conclusion.
\end{proof}

The next result refines Theorem 3.3 from Duval's work \cite{Duval1996}. It follows from Proposition \ref{PCM CM}, Remark \ref{9}, and the definition of sequentially Cohen--Macaulay simplicial complexes.

\begin{proposition}\label{SCM}
	A simplicial complex $\Delta$ is sequentially Cohen--Macaulay if and only if the pure skeleton $\Delta^{[i]}$ is Cohen--Macaulay for all $\mdim \Delta \le i \le \dim \Delta$.
\end{proposition}
\qed

The following propositions describe properties of truncated complexes.

\begin{proposition}
	Let $\Delta$ be a simplicial complex on $n$ vertices and $I=I_\Delta$. For any integer $k>\min\{\deg(u)\mid u\in I\}$, the truncated complex $\Delta_{I_k}$ is initially Cohen--Macaulay and
	$$\pdim \K[\Delta_{I_k}]=n-k+1.$$
\end{proposition}
\begin{proof}
	Set $i=n-k-1$. As $i<\dim\Delta^\vee$, \cite[Corollary 3.7]{Namiq2025} implies that the ideal $I_{(\Delta^\vee)^i}$ has a degree resolution. By Lemma \ref{Alexander dual},
	$$(\Delta^\vee)^i=(\Delta^\vee)^{n-k-1}=(\Delta_{I_{k}})^\vee.$$
	Therefore, Proposition \ref{PCM degree resolution} shows that $\Delta_{I_k}$ is initially Cohen--Macaulay. Furthermore, \cite[Proposition 2.2]{AhmedFrobergNamiq2023} gives $\bight I_{k}=n-k+1$. Since $\pdim\K[\Delta_{I_k}]=\bight I_k$ for initially Cohen--Macaulay complexes by Proposition \ref{pdim-bight}, the result follows.
\end{proof}

\begin{proposition}
	For a squarefree monomial ideal $I$, the following inequality holds:
	$$\dim\K[\Delta]-\dep\K[\Delta]\ge\bight I-\hte I.$$
	Equality holds if and only if $\Delta$ is initially Cohen--Macaulay.
\end{proposition}
\begin{proof}
	This is a direct consequence of Corollary \ref{ht-bight} and Proposition \ref{depth-mdim}.
\end{proof}

\begin{proposition}
	A simplicial complex $\Delta$ is initially Cohen--Macaulay if and only if $\K[\Delta^{\mdim \Delta}]$ has a unique extremal Betti number and
	$$\reg \K[\Delta^{\mdim \Delta}]=\deg h_{\K[\Delta^{\mdim \Delta}]}(t).$$
\end{proposition}
\begin{proof}
	If $\Delta$ is initially Cohen--Macaulay, Proposition \ref{PCM CM} states that $\Delta^{\mdim\Delta}$ is Cohen--Macaulay. By \cite[Corollary B.4.1]{Vasconcelos1998}, this implies $\K[\Delta^{\mdim\Delta}]$ has a unique extremal Betti number. From \cite[Proposition 2.5.2]{Namiq2023}, we have the relation
	$$\pdim\K[\Delta^{\mdim\Delta}]+\reg\K[\Delta^{\mdim\Delta}]=\deg h_{\K[\Delta^{\mdim\Delta}]}(t)+\hte I_{\Delta^{\mdim\Delta}}$$
	The result then follows from applying Corollary \ref{pdim-hte} since $\K[\Delta^{\mdim\Delta}]$ is Cohen--Macaulay.
\end{proof}

The following well-known lemma provides the topological interpretation necessary to generalize classical Cohen--Macaulay criteria.

\begin{lemma}\label{depth Hochster}
	Let $\Delta$ be a simplicial complex. The depth of its Stanley--Reisner ring is
	$$\dep\K[\Delta]=\min\{|F|+i+1\mid F\in\Delta,\ \widetilde{H}_i(\link_\Delta(F);\K)\ne0\}.$$
\end{lemma}
\begin{proof}
	This result follows from Grothendieck’s characterization of depth via local cohomology \cite[Corollary 3.10]{Grothendieck1967}, combined with Hochster’s formula \cite[Theorem 4.1]{Stanley1996} which expresses local cohomology modules using the reduced homology of links.
\end{proof}

\begin{proposition}[Generalized Reisner’s Criterion]\label{pro 1}
	A simplicial complex $\Delta$ is initially Cohen--Macaulay if and only if for every face $F \in \Delta$,
	$$\widetilde{H}_i(\link_\Delta(F); \K)=0\quad\text{for all }i<\mdim\Delta-|F|.$$
\end{proposition}
\begin{proof}
	The equality $\dep \K[\Delta]=\mdim \K[\Delta]$ holds if and only if the minimum value provided by Lemma \ref{depth Hochster} equals the initial dimension by Proposition \ref{depth-mdim}. This means
	$$\mdim \K[\Delta]=\min\left\{|F|+i+1\mid F\in\Delta,\ \widetilde{H}_i(\link_\Delta(F);\K)\neq 0\,\right\}.$$
	This condition is met if and only if $\widetilde{H}_i(\link_\Delta(F);\K)=0$ for all faces $F$ and all integers $i$ such that
	$$|F|+i+1<\mdim\K[\Delta].$$
	Since $\mdim\K[\Delta]=\mdim\Delta+1$, the inequality is equivalent to $i<\mdim\Delta-|F|$. Therefore, the vanishing condition becomes:
	$$\widetilde{H}_i(\link_\Delta(F);\K)=0\quad\text{for all }F\in\Delta\text{ and all } i<\mdim \Delta-|F|.$$
	Thus, $\Delta$ is initially Cohen--Macaulay if and only if the reduced homology of $\link_\Delta(F)$ vanishes for all $i<\mdim \Delta-|F|$, for every face $F\in\Delta$.
\end{proof}

\begin{corollary}
	A simplicial complex $\Delta$ is initially Cohen--Macaulay if and only if $\widetilde{H}_i(\link_\Delta(F); \K)=0$ for all faces $F\in\Delta$ with $|F|\le\mdim\Delta$ and for all $i<\mdim\Delta-|F|$.
\end{corollary}
\begin{proof}
	The forward direction is trivial by Proposition \ref{pro 1}.
	
	Conversely, assume the vanishing condition holds for all $F$ with $|F|\le\mdim\Delta$. The condition is trivially satisfied if $|F|>\mdim \Delta$, because the range $i<\mdim\Delta-|F|$ is empty (as $i \ge -1$). Thus, the condition holds for all faces $F\in\Delta$, and by Proposition \ref{pro 1}, $\Delta$ is initially Cohen--Macaulay.
\end{proof}

\begin{corollary}
	A simplicial complex $\Delta$ is initially Cohen--Macaulay if and only if $\link_\Delta(F)$ is initially Cohen--Macaulay for every face $F$ contained in a minimum facet of $\Delta$.
\end{corollary}
\begin{proof}
	First, suppose $\Delta$ is initially Cohen--Macaulay. Let $F$ be a face in a minimum facet of $\Delta$, and let $F'\in\link_\Delta(F)$. By Proposition \ref{pro 1},
	$$\widetilde{H}_i(\link_\Delta(F\cup F');\K)=0\quad\text{for all }i<\mdim\Delta-|F\cup F'|.$$
	Using the identity $\link_{\link_\Delta(F)}(F') = \link_\Delta(F\cup F')$ and the fact that $\mdim\link_\Delta(F)=\mdim\Delta-|F|$, the condition becomes
	$$\widetilde{H}_i(\link_{\link_\Delta(F)}(F');\K)=0\quad \text{for all } i<\mdim\link_\Delta(F)-|F'|.$$
	By Proposition \ref{pro 1}, this implies $\link_\Delta(F)$ is initially Cohen--Macaulay.
	
	Conversely, assume $\link_\Delta(F)$ is initially Cohen--Macaulay for all $F$ contained in a minimum facet. Taking the trivial face $F=\emptyset$ (which is in every facet), we have $\link_\Delta(\emptyset)=\Delta$. By hypothesis, $\Delta$ must be initially Cohen--Macaulay.
\end{proof}

A disconnected complex cannot be initially Cohen--Macaulay unless it has a 0-dimensional facet. Connectedness is therefore crucial, motivating the following definitions.

\begin{definition}
	Let $\Delta$ be a simplicial complex. We say that $\Delta$ is
	\begin{itemize}
		\item \emph{strongly connected} if for any two facets $F,F'$, there exists a sequence of facets $F=F_0,F_1,\dots,F_t=F'$ such that $|F_i\cap F_{i+1}|=\dim\Delta$ for all $i=0,\dots,t-1$ (see, e.g., \cite[Definition 11.6]{Bjorner1995}).
		
		\item \emph{weakly connected} if its initial-dimension skeleton $\Delta^{\mdim\Delta}$ is strongly connected.
		
		\item \emph{stably connected} if the pure skeleton $\Delta^{[i]}$ is strongly connected for every $-1\le i\le\dim\Delta$.
	\end{itemize}
\end{definition}

\begin{lemma}\label{PCM weakly connected}
	Let $\Delta$ be a simplicial complex.
	\begin{enumerate}
		\item If $\Delta$ is initially Cohen--Macaulay, then it is weakly connected.
		\item If $\Delta$ is sequentially Cohen--Macaulay, then it is stably connected.
	\end{enumerate}
\end{lemma}
\begin{proof}
	(1) This follows from Proposition \ref{PCM CM} and the fact that Cohen--Macaulay complexes are strongly connected \cite[Proposition 11.7]{Bjorner1995}.
	
	(2) This similarly follows from Proposition \ref{SCM} and \cite[Proposition 11.7]{Bjorner1995}.
\end{proof}

The converses of these implications are not true in general. For example, the independence complex of the cycle graph $C_7$ is weakly connected, but Proposition \ref{PCM Cn} shows it is not initially Cohen--Macaulay. However, the converse does hold for co-chordal graphs (see Proposition \ref{PCM co-chordal} and Corollary \ref{SCM stably connected}) and for complexes with $\mdim\Delta\le1$.

\begin{proposition}\label{mdim=1}
	Let $\Delta$ be a simplicial complex with $\mdim\Delta\le1$. Then $\Delta$ is initially Cohen--Macaulay if and only if it is weakly connected.
\end{proposition}
\begin{proof}
	The forward implication is given by Lemma \ref{PCM weakly connected}.
	
	To prove the converse, we analyze the two possible cases for $\mdim\Delta$.
	
	\medskip
	\noindent\textbf{Case 1:} $\mdim\Delta=1$. By Proposition \ref{pro 1}, $\Delta$ is initially Cohen--Macaulay if and only if $\widetilde{H}_i(\link_\Delta(F);\K)=0$ for all $i<1-|F| = -\dim F$. We examine this homological condition based on the dimension of the face $F$:
	
	\begin{enumerate}[\rmfamily (i)]
		\item If $F=\emptyset$ ($\dim F=-1$), the condition is $i<1$. This requires $\widetilde H_i(\Delta;\K)=0$ for $i=-1,0$. $\widetilde H_{-1}(\Delta;\K)=0$ means $\Delta\neq\emptyset$. $\widetilde H_0(\Delta;\K)=0$ means $\Delta$ is connected.
		
		\item If $F$ is a vertex ($\dim F=0$), the condition is $i<0$, so $i=-1$. This requires $\widetilde H_{-1}(\link_\Delta(F);\K)=0$, meaning $\link_\Delta(F)\neq\emptyset$. This is equivalent to $\Delta$ having no isolated vertices (every vertex belongs to a facet of dimension $\ge 1$, which must be 1 since $\mdim\Delta=1$).
		
		\item If $F$ is an edge ($\dim F=1$), the condition is $i<-1$, which is always true as reduced homology vanishes for $i<-1$.
	\end{enumerate}
	Taken together, $\Delta$ is initially Cohen--Macaulay if and only if it is non-empty, connected, and has no isolated vertices. These are precisely the conditions for $\Delta$ (with $\mdim\Delta=1$) to be weakly connected (i.e., $\Delta^1$ is strongly connected).
	
	\medskip
	\noindent\textbf{Case 2:} $\mdim\Delta=0$. If $F=\emptyset$, the condition is $i<0$, so $i=-1$. This requires $\widetilde H_{-1}(\Delta;\K)=0$, which just means $\Delta\neq\emptyset$. For any non-empty face $F$, the condition is $i<-1$, which is vacuous. Thus, $\Delta$ is initially Cohen--Macaulay if and only if $\Delta\neq\emptyset$, which is equivalent to $\Delta^0$ being strongly connected (i.e., weakly connected).
	
	In both cases, the converse holds.
\end{proof}

%%%%%%%%%%%%%%%%%%%%%%%%%%%%%%%%%%%%%%%%%%%%%%%%%%%%%%%%%%%%%%%%%%

\section{Applications in Graph Theory}\label{sec6}

\quad This part of the paper transitions from the general algebraic theory of initially Cohen--Macaulay modules to their specific applications in graph theory, focusing on the properties of independence complexes and edge ideals.

\subsection{The Initially Cohen--Macaulay Property of Graphs}

While many classical results for sequentially Cohen--Macaulay and Cohen--Macaulay complexes extend to the initially Cohen--Macaulay setting, a direct extension to edge ideals is not straightforward. The primary obstruction is that the initial-dimension skeleton, $\Delta^{\mdim\Delta}$, of an independence complex $\Delta_G$ does not necessarily correspond to the independence complex of any graph. This lack of correspondence complicates the translation between algebraic properties and their underlying graph-theoretic interpretations.

We define a graph $G$ as \emph{initially Cohen--Macaulay} over $\K$ if its Stanley--Reisner ring $\K[\Delta_G]$ (where $\Delta_G$ is the independence complex) is initially Cohen--Macaulay.

\begin{lemma}\label{indim P_n, C_n}
	Let $P_n$ be the path graph on $n \ge 2$ vertices and $C_n$ be the cycle graph on $n \ge 3$ vertices. Then
	$$\mdim R/I(P_n)=\mdim R/I(C_n)=\ceil{\frac{n}{3}}.$$
\end{lemma}
\begin{proof}
	Let $X=\{x_1,\dots,x_n\}$ be the vertex set of $G$. For small $n \le 5$, the claim is easily verified. For $n > 5$, one can construct an explicit maximal independent set $F_n$ of minimum cardinality based on $n \pmod 3$, namely
	$$F_n=\begin{cases} 
		\{x_1,x_4,\dots,x_{n-5},x_{n-2}\},& \text{ if }n=3k\\
		\{x_1,x_4,\dots,x_{n-3},x_{n-1}\},& \text{ if }n=3k+1\\
		\{x_1,x_4,\dots,x_{n-4},x_{n-1}\},& \text{ if }n=3k+2.
	\end{cases}$$
	For $P_n$, this set shows $\mdim R/I(P_n)=\ceil{\frac{n}{3}}$. Since this set $F_n$ remains a maximal independent set of minimum cardinality for $C_n$ as well, the same result holds.
\end{proof}

\begin{proposition}\label{PCM Pn}
	Every chordal graph is initially Cohen--Macaulay. In particular, for $n\ge2$, the path graph $P_n$ is initially Cohen--Macaulay and $\dep R/I(P_n)=\ceil{\frac{n}{3}}$.
\end{proposition}
\begin{proof}
	It is known that all chordal graphs are sequentially Cohen--Macaulay \cite[Theorem 3.2]{FranciscoTuyl2007}. The proposition follows as a direct consequence of Corollary \ref{SCM-implies-ICM} which states sequentially Cohen--Macaulay implies initially Cohen--Macaulay and Lemma \ref{indim P_n, C_n}.
\end{proof}

\begin{proposition}\label{PCM Cn}
	For $n\ge3$, the cycle graph $C_n$ is initially Cohen--Macaulay if and only if $n \not\equiv 1 \pmod 3$.
\end{proposition}
\begin{proof}
	This characterization is obtained by comparing the module's initial dimension with its depth. From Lemma \ref{indim P_n, C_n}, we have $\mdim R/I(C_n)=\ceil{\frac{n}{3}}$. The depth is known to be $\dep R/I(C_n)=\ceil{\frac{n-1}{3}}$ \cite[Proposition 1.3]{Cimpoeas2015}. The module is initially Cohen--Macaulay precisely when these two values are equal, which occurs if and only if $n \not\equiv 1 \pmod 3$.
\end{proof}

We now shift focus to the clique complex $\Delta(G)$ of a chordal graph.

\begin{proposition}\label{PCM co-chordal}
	Let $G$ be a chordal graph. Then its clique complex $\Delta(G)$ is initially Cohen--Macaulay if and only if it is weakly connected.
\end{proposition}
\begin{proof}
	The forward direction is given by Lemma \ref{PCM weakly connected}.
	
	Conversely, assume $G$ is chordal and $\Delta = \Delta(G)$ is weakly connected. By Fröberg’s theorem \cite[Theorem 1]{Froberg1990}, the ideal $I_\Delta$ has a linear resolution and $n-\mdim\Delta$ equals the maximal number of vertices in an induced disconnected subgraph of $G$. This implies $\pdim\K[\Delta]=n-\mdim\Delta-1=\bight I_\Delta$ \cite[Corollary 1.2]{Katzman2006}. By Proposition \ref{pdim-bight}, this equality is sufficient to show that $\Delta$ is initially Cohen--Macaulay.
\end{proof}

\begin{corollary}\label{CM co-chordal}
	For a chordal graph $G$, the clique complex $\Delta(G)$ is Cohen--Macaulay if and only if it is strongly connected.
\end{corollary}
\begin{proof}
	If $\Delta=\Delta(G)$ is strongly connected, it must be pure. This purity ensures that $\Delta$ is counmixed (all associated primes have the same coheight). The result then follows from Proposition \ref{ICM-unmixed} and Proposition \ref{PCM co-chordal}.
\end{proof}

Let $(d_1,\dots,d_q)$ be a non-increasing sequence of positive integers. A graph $G$ is a \emph{$(d_1,\dots,d_q)$-tree} if it is formed by recursively gluing complete graphs $\mathcal{K}_{d_i}$ to a previously formed graph $H_{i-1}$ along a common complete subgraph $\mathcal{K}_{d_i-1}$, starting from $H_1=\mathcal{K}_{d_1}$ \cite[Definition 3.1]{AhmedMafiNamiq2025}.

\begin{corollary}\label{SCM stably connected}
	Let $\Delta=\Delta(G)$ be the clique complex of a chordal graph $G$. The following are equivalent:
	\begin{enumerate}
		\item $G$ is a $(d_1, \dots, d_q)$-tree.
		\item $\Delta$ is stably connected.
		\item $\Delta$ is sequentially Cohen--Macaulay.
	\end{enumerate}
\end{corollary}
\begin{proof}
	We prove the implication (1) $\implies$ (2) by induction on $q$. The base case $q=1$ is clear, since $G=\mathcal{K}_{d_1}$ means $\Delta$ is a simplex, which is stably connected. For the inductive step, let $G$ be formed by gluing $\mathcal{K}_{d_q}$ to a $(d_1,\dots,d_{q-1})$-tree $H_{q-1}$ along $\mathcal{K}_{d_q-1}$. The clique complex $\Delta(G)$ decomposes as $\Delta' \cup \Delta''$, where $\Delta' = \Delta(H_{q-1})$ and $\Delta'' = \Delta(\mathcal{K}_{d_q})$. By induction, $\Delta'$ is stably connected. Since $\Delta''$ is a simplex of dimension $d_q-1$, it is also stably connected. Their intersection, $\Delta'\cap\Delta''=\Delta(\mathcal{K}_{d_q-1})$, is a simplex of dimension $d_q-2$ and is thus also stably connected. Therefore, $\Delta$ is stably connected.
	
	For (2) $\implies$ (3), if $\Delta$ is stably connected, its pure skeletons $\Delta^{[i]}$ are strongly connected. Since $G$ is chordal, these skeletons are also clique complexes of chordal graphs. By Corollary \ref{CM co-chordal}, they are Cohen--Macaulay. Thus, $\Delta$ is sequentially Cohen--Macaulay by Proposition \ref{SCM}.
	
	The implication (3) $\implies$ (1) is established in \cite[Theorem 3.2]{AhmedMafiNamiq2025}.
\end{proof}

Finally, we define a property that combines the initial Cohen--Macaulay condition with a homological property of the ideal itself.

\begin{definition}
	A graded ideal $I$ is \emph{bi-initially Cohen--Macaulay} if $R/I$ is initially Cohen--Macaulay and $I$ has a degree resolution. In particular, a simplicial complex $\Delta$ is bi-initially Cohen--Macaulay if both $\Delta$ and $\Delta^\vee$ are initially Cohen--Macaulay.
\end{definition}

\begin{proposition}\label{bi-PCM graphs}
	Let $\Delta=\Delta(G)$ be the clique complex of a graph $G$. Then $\Delta$ is bi-initially Cohen--Macaulay if and only if $G$ is chordal and $\Delta$ is weakly connected.
\end{proposition}
\begin{proof}
	This follows by combining three characterizations. First, $I_\Delta = I(\overline{G})$ has a linear resolution if and only if $G$ is chordal \cite[Theorem 1]{Froberg1990}. Second, by Proposition \ref{PCM degree resolution}, $I_{\Delta}$ has a degree resolution if and only if $\Delta^\vee$ is initially Cohen--Macaulay. Third, $\Delta$ is initially Cohen--Macaulay if and only if it is weakly connected by Proposition \ref{PCM co-chordal}.
\end{proof}

\subsection{Graphs with Projective Dimension Equal to Maximum Degree}

We conclude by addressing the problem of classifying graphs whose projective dimension equals their maximum vertex degree. This extends previous work, which was primarily focused on co-chordal graphs \cite{AhmedMafiNamiq2025, Froberg2022, GitlerValencia2005, MoradiKiani2010}, to a classification of all graphs with this property.

\begin{proposition}\label{big max}
	Let $G$ be a graph. The big height of its edge ideal, $\bight I(G)$, equals the maximum vertex degree of $G$ if and only if there exists a free vertex contained in a facet of the independence complex $\Delta_G$ of minimum cardinality.
\end{proposition}
\begin{proof}
	First, suppose $x$ is a free vertex in a minimum facet $F$ of $\Delta_G$. The set $C:=X(G)\setminus F$ is a minimal vertex cover. The freeness of $x$ implies it belongs to no other facet, which means $x$ must be adjacent to every vertex in $C$. Thus $N_G(x) = C$, and $\deg_G(x) = |C| = \bight I(G)$. Since $F$ is a minimum facet, $x$ has minimum degree in $\overline{G}$ and thus maximum degree in $G$.
	
	Conversely, suppose $\bight I(G) = \deg_G(x)$ for a vertex $x$ of maximum degree. Let $C = X(G) \setminus F$ for a minimum facet $F$. The condition implies $|C| = \bight I(G) = \deg_G(x) = |N_G(x)|$. It follows that $N_G(x) = C$, which means $x$ is not adjacent to any vertex in $F$ and adjacent to all vertices not in $F$. This implies $F$ is the only facet containing $x$, making $x$ a free vertex in a minimum facet.
\end{proof}

\begin{proposition}\label{pdim max}
	For any graph $G$, $\pdim \K[\Delta_G]$ equals the maximum vertex degree of $G$ if and only if $\Delta_G$ is initially Cohen--Macaulay and has a minimum facet containing a free vertex.
\end{proposition}
\begin{proof}
	This result is a direct combination of our previous findings. Proposition \ref{pdim-bight} states that $\pdim\K[\Delta_G]=\bight I(G)$ if and only if $\Delta_G$ is initially Cohen--Macaulay. Proposition \ref{big max} provides the combinatorial condition for when $\bight I(G)$ equals the maximum vertex degree.
\end{proof}

\begin{corollary}
	Let $G$ be a co-chordal graph. Then $\pdim \K[\Delta_G]$ equals its maximum vertex degree if and only if $\Delta_G$ is weakly connected and contains a minimum facet with a free vertex.
\end{corollary}
\begin{proof}
	This follows by applying Proposition \ref{pdim max} and then using Proposition \ref{PCM co-chordal} to replace the initially Cohen--Macaulay condition with the equivalent weakly connected condition for co-chordal graphs.
\end{proof}

\vspace{1ex}
\noindent\textbf{Concluding remark.}
We conclude by revisiting Hochster’s celebrated remark that ``life is really worth living in a Cohen--Macaulay ring" \cite[p. 887]{Hochster1978}. The results in this paper suggest that life is still very much worth living in the broader class of initially Cohen--Macaulay modules, given their rich and natural algebraic-combinatorial structure.

\bibliographystyle{amsrefs}
\bibliography{References}

\end{document}